\documentclass[11pt,reqno]{amsart}
\usepackage{amsmath}
\usepackage{cases}
\usepackage{mathrsfs}
\usepackage{bbm}
\usepackage{amssymb}
\usepackage{amscd}
\usepackage{amsfonts,latexsym,amsmath,amsthm,amsxtra,mathdots,amssymb,latexsym,mathabx}
\usepackage[all,cmtip]{xy}
\RequirePackage{amsmath} \RequirePackage{amssymb}
\usepackage{color}
\usepackage{colordvi}
\usepackage{multicol}
\usepackage{hyperref}
\usepackage{mathtools}
\usepackage[margin=1in]{geometry}
\usepackage{xcolor}
\hypersetup{
    colorlinks,
    linkcolor={red!50!black},
    citecolor={blue!50!black},
    urlcolor={blue!80!black}
}
\usepackage{cite}

     \newcommand{\BF}{{\mathbb {F}}}

     \newcommand{\BR}{{\mathbb {R}}}

     \newcommand{\BZ}{{\mathbb {Z}}}

     \newcommand{\CJ}{{\mathcal {J}}}
     \newcommand{\CL}{{\mathcal {L}}}
     \newcommand{\CN}{{\mathcal {N}}}

\def\-{^{-1}}

\newcommand{\bea}{\begin{eqnarray}}
\newcommand{\eea}{\end{eqnarray}}
\newcommand{\bna}{\begin{eqnarray*}}
\newcommand{\ena}{\end{eqnarray*}}

\numberwithin{equation}{section}

\newtheorem{Theorem}{Theorem}[section]
\newtheorem{Lemma}[Theorem]{Lemma}

\newtheorem{Proposition}[Theorem]{Proposition}
\newtheorem{Remark}[Theorem]{Remark}
\newcommand{\sumstar}{\sideset{}{^\star}\sum}

\begin{document}

\title{A subconvex bound for twisted $L$-functions}

\author[Qingfeng Sun and Hui Wang]{Qingfeng Sun and Hui Wang}

\address{School of Mathematics and Statistics
\\ Shandong University at Weihai \\ Weihai \\Shandong 264209 \\China}
\email{qfsun@sdu.edu.cn}

\address{School of Mathematics and Statistics
\\ Shandong University at Weihai \\ Weihai \\Shandong 264209 \\China}
\email{wh0315@mail.sdu.edu.cn}

\begin{abstract}
Let $\mathfrak{q}>2$ be a prime number, $\chi$ a primitive Dirichlet
character modulo $\mathfrak{q}$ and $f$ a primitive holomorphic cusp form or a Hecke-Maass cusp form
of level $\mathfrak{q}$
and trivial nebentypus. We prove the subconvex bound
$$
L(1/2,f\otimes \chi)\ll \mathfrak{q}^{1/2-1/12+\varepsilon},
$$
where the implicit constant depends only on the archimedean parameter of $f$ and $\varepsilon$.
The main input is a modifying
trivial delta method
developed in \cite{AHLS}.
\end{abstract}



\maketitle
\section{Introduction}

Let $f$ be a Hecke cusp form of level $\mathfrak{p}$ with trivial nebentypus
and let $\chi$ be a primitive Dirichlet character modulo $\mathfrak{q}$.
The twisted $L$-function $L(s,f\otimes \chi)$ is defined by
\bna
L(s,f\otimes \chi)=\sum_{n=1}^{\infty}\frac{\lambda_f(n)\chi(n)}{n^s}
\ena
which can be continued to be an entire function with a functional equation relating
$s$ and $1-s$.
The conductor of $L(s,f\otimes \chi)$ is
$\mathfrak{Q}(f\otimes \chi)=[\mathfrak{p},\mathfrak{q}^2]$ (see Proposition 14.19 in \cite{IK}),
where $[a,b]$ denotes the least common multiple of $a$ and $b$, and
the convexity bound is $\mathfrak{Q}(f\otimes \chi)^{1/4+\varepsilon}$
for any given $\varepsilon>0$. The subconvexity problem of $L(s,f\otimes \chi)$
was first studied in Duke, Friedlander and Iwaniec \cite{DFI} for $\mathfrak{p}=1$,
in which case the current record result
is $\mathfrak{q}^{3/8+\varepsilon}$
(the so-called Burgess bound) established in Bykovski\u{i} \cite{B},
 Blomer, Harcos and Michel \cite{BHM} and Blomer and Harcos \cite{BH} successively.
Compared with the beautiful and systematic results in this case, however, the results
in the case that both $\mathfrak{p}$ and $\mathfrak{q}$ are varying (i.e., the hybrid bound),
are not so good though there are many studies. For example, for $(\mathfrak{p},\mathfrak{q})=1$,
subconvex bounds are only known for $\mathfrak{p}$ and $\mathfrak{q}$ in a hybrid range, i.e.,
$\mathfrak{p}=\mathfrak{q}^{\eta}$ for some $\eta\in (0,1)$
(see Blomer and Harcos \cite{BH}, Aggarwal, Jo and Nowland \cite{AJN} and
Hou and Chen \cite{HC}).
For the special case $\mathfrak{p}|\mathfrak{q}$, $\mathfrak{q}$ odd and squarefree,
and $\chi$ being a primitive real character,
where the positivity property $L(s,f\otimes \chi)\geq 0$ is known,
Conrey and Iwaniec \cite{CI} established the Weyl-type subconvexity bound
$\mathfrak{q}^{1/3+\varepsilon}$. This was generalized to
arbitrary characters by Petrow and Young (see \cite{PY2} and \cite{PY3}).

It is interesting to try different methods to continue studying the case
that both $\mathfrak{p}$ and $\mathfrak{q}$ are varying.
In \cite{AHLS}, for $\mathfrak{p}=1$ and $\mathfrak{q}$ prime,
Aggarwal, Holowinsky, Lin and the first author proved
the Burgess bound $L(s,f\otimes \chi)\ll_{f, \varepsilon} \mathfrak{q}^{3/8+\varepsilon}$.
In this paper, we will show how using a modifying trivial delta method
developed in \cite{AHLS} to prove a subconvex bound for
$L(1/2,f\otimes \chi)$ in the case $\mathfrak{p}=\mathfrak{q}$ and $\mathfrak{q}$ prime.
Our main result is the following.
\begin{Theorem}\label{main thm}
Let $\mathfrak{q}$ be prime.
Let $f$ be a fixed Hecke cusp form for $\rm \Gamma_0(\mathfrak{q})$ and
$\chi$ be a primitive Dirichlet character modulo $\mathfrak{q}$. For any $\varepsilon>0$,
\begin{equation*}
\begin{split}
L\left(\frac{1}{2},f\otimes \chi\right)\ll_{f,\varepsilon}
\mathfrak{q}^{1/2-1/12+\varepsilon},
\end{split}\end{equation*}
where the implicit constant depends only on the archimedean parameter of $f$ and $\varepsilon$.
\end{Theorem}

\begin{Remark}
Fouvry, Kowalski and Michel \cite{FKM} pointed that by choosing an appropriate
amplifier, their methods would
give the exponent $1/2-1/16+\varepsilon$ in Theorem 1.1 (see Remark 4.5 in \cite{FKM}).
Note that the conductor of $L\left(s,f\otimes \chi\right)$
is $\mathfrak{q}^2$ and $\mathfrak{q}^{1/2-1/12}=\mathfrak{q}^{2(1/4-1/24)}$.
The $1/24$-saving is consistent with the bound in \cite{IS}
in the sup-norm problem. This is not a coincidence as we have used the amplification
technique based on the Hecke relation $\lambda_f^2(\ell)=1+\lambda_f(\ell^2)$ ($\ell$ prime),
in which situation the generic $1/24$-saving
is the natural limit. For example, for $g$ a cuspidal
holomorphic newform of even positive weight or a cuspidal Maass newform of exact
level $q$ with $q$ square-free coprime to 6, Blomer and Khan \cite{BK} proved the subconvexity
bound $L(1/2,g)\ll q^{1/4-(1-2\vartheta)/24+\varepsilon}$, where $\vartheta$ is the exponent towards
the Ramanujan conjecture for $g$. We note that the bound in Theorem 1.1
does not depend on the Ramanujan conjecture.
\end{Remark}

\begin{Remark}
Let $\pi$ be a fixed cuspidal representation of $\rm GL_2$ over an arbitrary number field
$\mathbf{F}$ and $\chi$ be a varying Hecke
character.
Wu \cite{Wu} established a Burgess-like subconvex bound for $L(1/2,\pi\otimes \chi)$
with an explicit dependence on the analytic conductor $C(\pi)$ of $\pi$.
\end{Remark}

The trivial delta method is based on the following trivial key identity,
\bea\label{trivial delta}
\delta(n)=\frac{1}{q}\sum_{c|q}\underset{\begin{subarray}{c}
a\bmod c\\ (a,c)=1 \end{subarray}}{\sum}e\left(\frac{an}{c}\right), \quad \mbox{when}\quad  q>|n|,
\eea
where $\delta(n)$ denotes the Kronecker delta symbol.
The good point of the trivial delta method approach in \cite{AHLS} is that both
the treatment of shifted convolution sums and the use of the
Petersson/Kuznetsov formula are avoided. For the sake of exposition
we assume $f$ is a Hecke-Maass cusp form with Laplace eigenvalue $1/4+\mu_f^2$. The
holomorphic case is very similar and more simpler.
As we know, the approximate functional equation and the Rankin-Selberg estimate
(see Proposition 19.6 in \cite{DFI2})
\bea\label{Rankin-Selberg}
\sum_{n\leq X}|\lambda_f(n)|^2\ll_{\varepsilon}X(X\mathfrak{q}|\mu_f|)^{\varepsilon}
\eea
imply that
\bea\label{to S}
L\left(\frac{1}{2},f\otimes \chi \right)\ll_{\varepsilon} \mathfrak{q}^{\varepsilon}
\sup_{N}\frac{|\mathcal{S}(N)|}{\sqrt{N}}+\mathfrak{q}^{\beta/2+\varepsilon},
\eea
where the supremum is taken over $N$ in the range
$\mathfrak{q}^\beta<N<\mathfrak{q}^{1+\varepsilon}$ with $\beta>2/3$, and
\begin{equation*}
\begin{split}
\mathcal{S}(N)=\sum_{n=1}^{\infty}\lambda_f(n)\chi(n)W\left(\frac{n}{N}\right).
\end{split}
\end{equation*}
Here $W$ is a smooth bump function supported on $[1,2]$ with $W^{(j)}(x)\ll_j 1$.
The trivial bound of $\mathcal{S}(N)$ is $N^{1+\varepsilon}$ by Cauchy-Schwarz and
\eqref{Rankin-Selberg}. This is sufficient for our purpose for small $N$.
For larger $N$, we need
the delta symbol to be used as a device to separate the oscillations from
$\lambda_f(n)$ and $\chi(n)$. To break the convexity, an amplification technique is also
needed. If we proceed as in \cite{AHLS}, then $\mathcal{S}(N)$
will be transformed into the form
\bna
&&\left(\sum\limits_{\ell\in \mathcal{L}}|\lambda_f(\ell)|^2\right)^{-1}\sum_{\ell\in\mathcal{L}}\overline{\lambda_f(\ell)}
\sum_{n}\lambda_f(n)W\left(\frac{n}{N\ell}\right)
\sum_{m}\chi(m)V\left(\frac{m}{N}\right)\delta(n-m\ell)
+\mathbf{R},
\ena
where $\mathcal{L}$ denotes the set of primes $\ell$ in the dyadic interval $[L,2L]$ with
$L<\mathfrak{q}^{1/2}$ a parameter to be determined later,
$V$ is a smooth function supported on $[1/2,3]$,
constantly $1$ on $[1,2]$ and satisfying $V^{(j)}(x)\ll_j 1$,
and
\bna
\mathbf{R}=\left(\sum\limits_{\ell\in \mathcal{L}}|\lambda_f(\ell)|^2\right)^{-1}\sum\limits_{\ell\in \mathcal{L}}
\overline{\lambda_f(\ell)}
\sum_{m}\lambda_f\left(\frac{m}{\ell}\right)\chi(m)V\left(\frac{m}{N}\right).
\ena
The error term $\mathbf{R}$ arises
from the Hecke relation
\bea\label{hecke relation}
\lambda_f(m)\lambda_f(\ell)=\sum_{d|(m,\ell)}\lambda_f\left(\frac{m\ell}{d^2}\right)
 \qquad \mbox{for} \, (m\ell,\mathfrak{q})=1.
\eea
In \cite{AHLS}, $\mathbf{R}$ was bounded by the
prime number theorem for automorphic $L$-functions which states that (see Theorem 5.13 in \cite{IK})
\bna
\sum\limits_{\ell\in \mathcal{L}}|\lambda_f(\ell)|^2\asymp_f \frac{L}{\log L}.
\ena
However, in the case $f$ is of level $\mathfrak{q}$, the above bound can not be applied
since it depends on $f$. The main point of this paper is to demonstrate that a direct use of
the Hecke relation in \eqref{hecke relation} is sufficient to work with
the trivial delta method to produce a subconvex bound. More precisely,
denote $L^\star=\sum\limits_{\ell\in \mathcal{L}}1$.
Then $L^\star\asymp L/\log L$. By \eqref{hecke relation} we write
\bea\label{S(N)}
\mathcal{S}(N)
=\frac{1}{L^\star}\sum_{\ell\in\mathcal{L}}\sum_{n}\lambda_f(n)\chi(n)W\left(\frac{n}{N}\right)
\left(\lambda_f^{2}(\ell)-\lambda_f\left(\ell^2\right)\right)
=\mathcal{S}_1(N)-\mathcal{S}_2(N),
\eea
where
\bna
\mathcal{S}_1(N)=\frac{1}{L^\star}\sum_{\ell\in\mathcal{L}}
\lambda_f(\ell)
\sum_{n}\lambda_f(n)\lambda_f(\ell)\chi(n)
W\left(\frac{n}{N}\right),
\ena
and
\bna
\mathcal{S}_2(N)=\frac{1}{L^\star}\sum_{\ell\in\mathcal{L}}
\sum_{n}\lambda_f(n)\lambda_f(\ell^2)\chi(n)W\left(\frac{n}{N}\right).
\ena
In order to express
$\mathcal{S}_1(N)$ and $\mathcal{S}_2(N)$ in the form of being ready
to apply the trivial
key identity in \eqref{trivial delta}, we use the Hecke relation
in \eqref{hecke relation} and
the Rankin-Selberg estimate in \eqref{Rankin-Selberg} to write
\bea
\mathcal{S}_1(N)=\mathcal{S}_1^{\sharp}(N)
+O\left(\frac{N^{1+\varepsilon}}{L}\right),
\eea
and
\bea\label{S_2(N)}
\mathcal{S}_2(N)=\mathcal{S}_2^{\sharp}(N)
+O\left(\frac{N^{1+\varepsilon}}{L}\right),
\eea
where $L\geq N^{\vartheta}$ with $\vartheta$ being the exponent towards
the Ramanujan conjecture for $f$ (we can take $\vartheta=7/64$ by \cite{KS}), $U$ is a smooth function supported on $[1/2,9]$,
constantly $1$ on $[1,8]$ and satisfies $U^{(j)}(x)\ll_j 1$, and
\bna
\mathcal{S}_1^{\sharp}(N)=\frac{1}{L^\star}\sum_{\ell\in\mathcal{L}}\lambda_f(\ell)\sum_m
\lambda_f(m)U\left(\frac{m}{NL}\right)\sum_n
\chi(n)W\left(\frac{n}{N}\right)\delta(m-n\ell),
\ena
\bna
\mathcal{S}_2^{\sharp}(N)=\frac{1}{L^\star}\sum_{\ell\in\mathcal{L}}\sum_m
\lambda_f(m)U\left(\frac{m}{NL^2}\right)\sum_n
\chi(n)W\left(\frac{n}{N}\right)\delta(m-n\ell^2).
\ena

Following closely \cite{AHLS}, we can now apply the trivial
key identity in \eqref{trivial delta} and prove the following results.
\begin{Proposition}\label{cusp form prop}
Let $\mathfrak{q}^\beta<N<\mathfrak{q}^{1+\varepsilon}$ with $\beta>2/3$.
Then for $\mathfrak{q}^{1-\beta+\varepsilon}\ll L\ll \mathfrak{q}^{1/6-\varepsilon}$
we have
\bna
\frac{\mathcal{S}^{\sharp}_1(N)}{N^{1/2}}\ll N^\varepsilon\left(
\frac{\mathfrak{q}^{1/2}}{L^{1/2}} +\mathfrak{q}^{1/4}L^{1/2}
\right)
\ena
and
\bna
\frac{\mathcal{S}_2^\sharp(N)}{N^{1/2}}\ll N^\varepsilon
\left(\frac{\mathfrak{q}^{1/2}}{L^{1/2}} +\mathfrak{q}^{1/4}L\right)
\ena
\end{Proposition}

Now we return to the proof of Theorem 1.1.
Take $L=\mathfrak{q}^{1/6+\varepsilon}$ and $\beta=5/6+\varepsilon$.
By Proposition 1.2, \eqref{to S} and \eqref{S(N)}-\eqref{S_2(N)}, we conclude that
\bna
L\left(\frac{1}{2}, f\otimes \chi\right)\ll  \mathfrak{q}^{1/2-1/12+\varepsilon}.
\ena
This proves Theorem 1.1.

\begin{Remark}
One may try proving a lower bound for $\sum\limits_{\ell\in \mathcal{L}}|\lambda_f(\ell)|^2$
which does not depend on the level of $f$. Then the trivial delta method in \cite{AHLS}
may be also applied to the level aspect. However, this approach seems very difficult
if $L$ is small compared to the level of $f$. For example, Khan \cite{K} showed that an assumption
of the form $\sum\limits_{\ell\in \mathcal{L}}|\lambda_f(\ell)|^2>_{\epsilon}L^{1-\epsilon}$
for $L> \mathfrak{q}^{1/4+\eta}$ for any fixed $\eta>0$ yields a subconvexity bound for some
$\rm GL_3$$\times$$\rm GL_2$ $L$-functions in the $\mathfrak{q}$-aspect.

\end{Remark}

In subsequent sections, we are devoted to the proof of Proposition 1.2.
Since the analysis method of $\mathcal{S}_1^{\sharp}(N)$ is similar to that of $\mathcal{S}_2^{\sharp}(N)$, we will only analyze $\mathcal{S}_2^{\sharp}(N)$.

\medskip

\section{Proof of Proposition 1.2}

The proof of Proposition 1.2 is very similar to that in \cite{AHLS},
and we follow closely the proof there.
Let $P$ be a parameter and $\mathcal{P}=\{p\in [P,2P]| \;p \,\mbox{prime}\}$.
Denote $P^\star=\sum\limits_{p\in \mathcal{P}}1\asymp P/\log P$.
Let $p\in \mathcal{P}$, $m\asymp NL^2$ and $n\asymp N$.
For any given $\varepsilon>0$ and $P\mathfrak{q}\gg N^{1+\varepsilon}L^2$,
the condition $m=n\ell^2$ is equivalent to the congruence $m\equiv n\ell^2\bmod p\mathfrak{q}$.
Thus we assume
\begin{equation}\label{restriction 1}
P=N^{1+\varepsilon}L^2/\mathfrak{q}, \qquad L\gg \sqrt{\mathfrak{q}^{1+\varepsilon}/N}.
\end{equation}
Moreover, we assume $\mathcal{P}\cap \mathcal{L}=\emptyset$. Applying \eqref{trivial delta} with $q=p\mathfrak{q}$, we have
\begin{equation}\label{start}
\begin{split}
\mathcal{S}_2^{\sharp}(N)=\frac{1}{\mathfrak{q}L^\star P^\star}\sum_{\ell\in\mathcal{L}}\sum_{p\in\mathcal{P}}
\frac{1}{p}\sum_{c|p\mathfrak{q}}\quad\sideset{}{^*}\sum_{\alpha\bmod c}
\sum_n\chi(n)e\left(-\frac{\alpha n\ell^2}{c}\right)W\left(\frac{n}{N}\right)
\sum_m\lambda_f(m)e\left(\frac{\alpha m}{c}\right)
U\left(\frac{m}{NL^2}\right),
\end{split}
\end{equation}
where the $^*$ denotes the condition $(\alpha,c)=1$.

\subsection{Voronoi and Poisson summation formulas}

Let $f$ be a Hecke-Maass cusp form with Laplace eigenvalue $1/4+\mu_f^2$.
We have the following Voronoi formula for $f$ (see Theorem A.4 in \cite{KMV}).
\begin{Lemma}\label{voronoi}
For $(\alpha,c)=1$, set $\mathfrak{q}_1=(c,\mathfrak{q})$,
$\mathfrak{q}_2=\frac{\mathfrak{q}}{\mathfrak{q}_1}$ and assume that
$(\mathfrak{q}_1,\mathfrak{q}_2)=1$.
For $F\in C^{\infty}(\mathbb{R}^{+})$ a smooth function vanishing
in a neighbourhood of zero and rapidly decreasing,
\begin{equation*}\label{voronoi for tau}
\begin{split}
\sum_{m\geq 1}\lambda_f(m)e\left(\frac{\alpha m}{c}\right)F(m)
=
\frac{\eta_f(\mathfrak{q}_2)}{c\sqrt{\mathfrak{q}_2}}\sum_{\pm}
\sum_{m\geq 1}\lambda_{f\mathfrak{q}_2}(m)e\left(\mp\frac{\overline{\alpha\mathfrak{q}_2}m}{c}\right)
F^{\pm}\left(\frac{m}{\mathfrak{q}_2c^2}\right),
\end{split}
\end{equation*}
where $|\eta_f(\mathfrak{q}_2)|=1$ (in particular, $\eta_f(1)=1$),
\begin{equation*}
\lambda_{f\mathfrak{q}_2}(m)=
\begin{cases}
\lambda_f(m), \quad &\textit{ if } (m,\mathfrak{q}_1)=1,\\
\overline{\lambda_f(m)}, \quad &\textit{ if } m|\mathfrak{q}_1^\infty,
\end{cases}
\end{equation*}
and
$
F^{\pm}(y)=\int_0^\infty F(x)J_f^{\pm}\left(4\pi \sqrt{xy}\right)\mathrm{d}x
$
with
\begin{equation*}
J_f^{+} (x)= \frac{-\pi}{\sin(\pi i\mu_f)} \left(J_{2i\mu_f}(x) - J_{-2i\mu_f}(x)\right)
\end{equation*}
and
\begin{equation*}
J_f^{-} (x)= 4\varepsilon_f\cosh(\pi \mu_f)K_{2i\mu_f}(x) .
\end{equation*}
Here $\varepsilon_f$ be an eigenvalue of $f$ under the reflection operator.
Moreover,
$$F^\pm(y)\ll_A (1+|y|)^{-A}
$$
for any $A>0$.
\end{Lemma}

First, we apply Lemma \ref{voronoi} with $F(x)=U\left(x/NL^2\right)$
to the $m$-sum in \eqref{start} to get
\bea\label{after voronoi1}
m\mbox{-sum}=\eta_f(\mathfrak{q}_2)\frac{NL^2}{c\sqrt{\mathfrak{q}_2}}\sum_{\pm}
\sum_{m=1}^\infty\lambda_{f\mathfrak{q}_2}(m)e\left(\mp\frac{\overline{\alpha \mathfrak{q}_2}m}{c}\right)
U^{\pm}_f\left(\frac{mNL^2}{\mathfrak{q}_2c^2}\right),
\eea
where $U_f^{\pm} (y)=
\int_0^\infty U(x)J_f^{\pm}(4\pi\sqrt{xy})\ll_A(1+|y|)^{-A}$.
Thus $m$-sum is essentially supported in
$1\leq m\ll c^2\mathfrak{q}_2N^\varepsilon/N L^2$.

Next, we apply Poisson summation to the $n$-sum in \eqref{start} to get
\bea\label{after voronoi2}
n\mbox{-sum}=
\frac{N}{[c,\mathfrak{q}]} \sum_{n\in\BZ}\left(\sum_{\beta\bmod[c,\mathfrak{q}]}\chi(\beta)
e\left(\frac{-\alpha\beta\ell^2}{c}\right)
e\left(\frac{n\beta}{[c,\mathfrak{q}]}\right)\right)
\widehat{W}\left(\frac{nN}{[c,\mathfrak{q}]}\right),
\eea
where $\widehat{W}(y)=\int_{\mathbb{R}}W(x)e(-xy)\mathrm{d}x$ is the Fourier transform
of $W$. By the rapid decay of $\widehat{W}$, we can truncate the
$n$-sum at $|n|\ll  [c,\mathfrak{q}]N^{\varepsilon}/N$.
Here
as in \cite{AHLS}, we denote $a_b=a/(a,b)$,
where $(a,b)$ is the gcd of $a$ and $b$, and
$[a,b]$ denotes the lcm of $a$ and $b$.
Note that the character sum over $\beta$ vanishes unless
$n-\alpha\ell^2 \mathfrak{q}_c\equiv0\bmod c_\mathfrak{q}$,
in which case it equals
$c_\mathfrak{q}\overline{\chi}((n-\alpha\ell^2 \mathfrak{q}_c)
\overline{c_\mathfrak{q}})g_\chi$, where $g_{\chi}$ denotes the Gauss sum.

Substituting \eqref{after voronoi1} and \eqref{after voronoi2} into
\eqref{start}, we arrive at
\begin{equation}\label{after dual summations}
\begin{split}
\mathcal{S}_2^{\sharp}(N)=&\frac{1}{\mathfrak{q}L^\star P^\star}\sum_{\ell\in\mathcal{L}}\sum_{p\in\mathcal{P}}
\frac{1}{p}\sum_{c|p\mathfrak{q}}\quad\sideset{}{^*}\sum_{\alpha\bmod c}
\bigg(\frac{Ng_\chi}{\mathfrak{q}}\underset{\begin{subarray}{c}
|n|\ll [c,\mathfrak{q}]N^{\varepsilon}/N \\ n-\alpha\ell^2 \mathfrak{q}_c\equiv0\bmod c_\mathfrak{q}
\end{subarray}}{\sum}\overline{\chi}((n-\alpha\ell^2 \mathfrak{q}_c)\overline{c_\mathfrak{q}})
\widehat{W}\left(\frac{nN}{[c,\mathfrak{q}]}\right)\bigg)\\
\times&\left(\eta_f(\mathfrak{q}_2)\frac{NL^2}{c\sqrt{\mathfrak{q}_2}}\sum_{\pm}
\sum_{1\leq m\ll c^2\mathfrak{q}_2N^\varepsilon/N L^2}\lambda_{f\mathfrak{q}_2}(m)e\left(\mp\frac{\overline{\alpha \mathfrak{q}_2}m}{c}\right)
U^{\pm}_f\left(\frac{mNL^2}{\mathfrak{q}_2c^2}\right)\right)
+O_A(N^{-A}).
\end{split}\end{equation}

Note that $c|p\mathfrak{q}$ implies that $c=1, p, \mathfrak{q}$ or $p\mathfrak{q}$.
For $c=1$, its contribution to \eqref{after dual summations} is
\begin{equation*}
\begin{split}
&\frac{1}{\mathfrak{q}L^\star P^\star}\sum_{\ell\in\mathcal{L}}\sum_{p\in\mathcal{P}}\frac{1}{p}
\frac{Ng_\chi}{\mathfrak{q}}\sum_{|n|\ll \mathfrak{q}N^{\varepsilon}/N}\overline{\chi}(n-\alpha\ell^2 \mathfrak{q})
\widehat{W}\left(\frac{nN}{\mathfrak{q}}\right)\\
&\qquad \times \eta_f(\mathfrak{q})\frac{NL^2}{\sqrt{\mathfrak{q}}}\sum_{\pm}
\sum_{1\leq m\ll \mathfrak{q}N^{\varepsilon}/NL^2}
\lambda_{f}(m)U_f^{\pm}\left(\frac{mNL^2}{\mathfrak{q}}\right)
\ll\frac{N^\varepsilon}{P}.
\end{split}
\end{equation*}
For $c=p$, its contribution to \eqref{after dual summations} is
\begin{equation*}
\begin{split}
&\frac{1}{\mathfrak{q}L^\star P^\star}\sum_{\ell\in\mathcal{L}}\sum_{p\in\mathcal{P}}\frac{1}{p}
\quad\sumstar_{\alpha\bmod p}\frac{Ng_\chi}{\mathfrak{q}}\underset{\begin{subarray}{c}
|n|\ll P\mathfrak{q}N^{\varepsilon}/N\\ n\equiv\alpha\ell^2 \mathfrak{q}(\bmod p)
\end{subarray}}
{\sum}\overline{\chi}((n-\alpha\ell^2 \mathfrak{q})\overline{p})\widehat{W}\left(\frac{nN}{p\mathfrak{q}}\right)\\
&\times\eta_f(\mathfrak{q})\frac{NL^2}{p\sqrt{\mathfrak{q}}}\sum_{\pm}
\sum_{1\leq m\ll P^2\mathfrak{q}N^{\varepsilon}/NL^2}
\lambda_f(m)e\left(\mp\frac{\overline{\alpha \mathfrak{q}}m}{p}\right)
U_f^{\pm}\left(\frac{mNL^2}{p^2\mathfrak{q}}\right)
\ll PN^\varepsilon.
\end{split}
\end{equation*}
Combining these estimates with \eqref{after dual summations}, we obtain,
\begin{equation}\label{S(N) and S(N) star}
\mathcal{S}_2^{\sharp}(N) = \mathcal{S}_2^\star(N) +\mathcal{S}_2(c=\mathfrak{q})+
O\left( PN^{\varepsilon}\right),
\end{equation}
where
\begin{equation*}
\begin{split}
\mathcal{S}_2(c=\mathfrak{q})=&\frac{N^2L^2g_\chi}{\mathfrak{q}^3L^\star P^\star}\sum_{\pm}\sum_{\ell\in\mathcal{L}}
\sum_{p\in\mathcal{P}}\frac{1}{p}\sum_{|n|\leq N^{\varepsilon}\mathfrak{q}/N}
\widehat{W}\left(\frac{nN}{\mathfrak{q}}\right)\sum_{1\leq m\ll N^{\varepsilon}\mathfrak{q}^2/NL^2}
\lambda_f(m)U_f^{\pm}\left(\frac{mNL^2}{\mathfrak{q}^2}\right)\\
&\quad\sideset{}{^*}\sum_{\alpha\bmod \mathfrak{q}}\overline{\chi}(n-\alpha\ell^2)
e\left(\mp\frac{\overline{\alpha}m}{\mathfrak{q}}\right)
\end{split}
\end{equation*}
and
\begin{equation*}
\begin{split}
\mathcal{S}_2^\star(N)=&\frac{N^2L^2g_{\chi}}{L^\star P^\star \mathfrak{q}^3}\sum_{\pm}
\sum_{1\leq m\ll N^{\varepsilon}P^2\mathfrak{q}^2/NL^2}\lambda_f(m)\sum_{\ell\in\mathcal{L}}
\sum_{p\in\mathcal{P}}\frac{\chi(p)}{p^2}\\
&\times\sum_{
1\leq |n|\ll N^{\varepsilon}P\mathfrak{q}/N\atop(n,p)=1}
\mathfrak{D}(\pm m,n,\ell,p)e\left(\mp\frac{\overline{n\mathfrak{q}}m\ell^2}{p}\right)
\widehat{W}\left(\frac{nN}{p\mathfrak{q}}\right)U_f^{\pm}
\left(\frac{mNL^2}{p^2\mathfrak{q}^2}\right)
\end{split}
\end{equation*}
with
\begin{equation}\label{D}
\begin{split}
\mathfrak{D}(m,n,\ell,p)=\sideset{}{^*}\sum_{\alpha\bmod \mathfrak{q}}\overline{\chi}(n+\alpha)
e\left(\frac{\overline{\alpha p}m\ell^2}{\mathfrak{q}}\right).
\end{split}
\end{equation}
The estimation of $\mathcal{S}_2(c=\mathfrak{q})$ is similar as that of
$\mathcal{S}_2^\star(N)$ and much more simpler. In the following, we only
estimate $\mathcal{S}_2^\star(N)$ in details.

\subsection{Cauchy-Schwarz and Poisson Summation}

Breaking the $m$-sum into dyadic segments of length $\CN_0$ and
applying Cauchy-Schwarz inequality, we have
\begin{equation*}
\begin{split}
\mathcal{S}_2^\star(N)
\ll& \frac{N^{2+\varepsilon}L}{P\mathfrak{q}^{5/2}}\sum_{\pm}
\underset{\begin{subarray}{c}1\leq\CN_0\ll N^{\varepsilon}P^2\mathfrak{q}^2/NL^2
\\ \textit{dyadic}\end{subarray}}{\sum}
\CN_0^{1/2}\mathcal{S}_2^\star(N,\CN_0)^{1/2},
\end{split}
\end{equation*}
where
\begin{equation}\label{S star}
\begin{split}
\mathcal{S}_2^\star(N,\CN_0)=&\sum_{m} \widetilde{V}\left(\frac{m}{\CN_0}\right)
\bigg|\sum_{\ell\in\mathcal{L}}\sum_{p\in\mathcal{P}}\frac{\chi(p)}{p^2}
\underset{\begin{subarray}{c}1\leq |n|\ll N^{\varepsilon}P\mathfrak{q}/N
\\ (n,p)=1\end{subarray}}{\sum}\mathfrak{D}(\pm m,n,\ell,p)\\
&\times e\left(\mp\frac{\overline{n\mathfrak{q}}m\ell^2}{p}\right)
\widehat{W}\left(\frac{nN}{p\mathfrak{q}}\right)
U_f^{\pm}\left(\frac{mNL^2}{p^2\mathfrak{q}^2}\right)\bigg|^2
\end{split}
\end{equation}
with $\widetilde{V}(y)$ a smooth positive function, $\widetilde{V}(y)\equiv 1$ if
$y\in [1,2]$.
Here we have used the Rankin-Selberg estimate in \eqref{Rankin-Selberg}.

Note that by the square-root estimate of
the character sum $\mathfrak{D}$ in \eqref{D} $\mathfrak{D}\ll \sqrt{\mathfrak{q}}$
(see \cite{AHLS}), we have the trivial bound
\begin{equation*}
\begin{split}
\mathcal{S}_2^\star(N,\CN_0)\ll N^{\varepsilon}L^2\mathfrak{q}^3\CN_0/N^2,
\end{split}
\end{equation*}
Thus for $\mathcal{N}_0\ll P^2\mathfrak{q}^{1+\varepsilon}/N $, we have
$\mathcal{S}_2^\star(N)\ll N^\varepsilon PL^2\ll P^{2+\varepsilon}$ by
\eqref{restriction 1}. We will choose $P<\mathfrak{q}^{3/8}$. Then
\begin{equation}\label{S star final}
\mathcal{S}_2^\star(N)\ll \mathfrak{q}^{3/4+\varepsilon}
+\frac{N^{2+\varepsilon}L}{P\mathfrak{q}^{5/2}}\sum_{\pm}
\underset{\begin{subarray}{c}P^2\mathfrak{q}^{1+\varepsilon}/N\leq\CN_0\ll N^{\varepsilon}P^2\mathfrak{q}^2/NL^2
\\ \textit{dyadic}\end{subarray}}{\sum}\CN_0^{1/2}\mathcal{S}_2^\star(N,\CN_0)^{1/2}.
\end{equation}

Opening the square in \eqref{S star} and switching the order of summations,
we get
\begin{equation}\label{before truncate}
\begin{split}
\mathcal{S}_2^\star(N, \mathcal{N}_0)&=\sum_{\ell_1\in\mathcal{L}}\sum_{\ell_2\in\mathcal{L}}
\sum_{p_1\in\mathcal{P}}\sum_{p_2\in\mathcal{P}}\frac{\chi(p_1)\overline{\chi}(p_2)}{(p_1p_2)^2}
\underset{\begin{subarray}{c}1\leq |n_1|\ll N^{\varepsilon}P\mathfrak{q}/N \\ (n_1,p_1)=1\end{subarray}}{\sum}\quad
\underset{\begin{subarray}{c}1\leq |n_2|\ll N^{\varepsilon}P\mathfrak{q}/N \\ (n_2,p_2)=1\end{subarray}}{\sum}\\
&\widehat{W}\left(\frac{n_1N}{p_1\mathfrak{q}}\right)\overline{\widehat{W}\left(\frac{n_2N}{p_2\mathfrak{q}}\right)}
\sideset{}{^*}\sum_{\alpha_1\bmod \mathfrak{q}}\overline{\chi}(n_1+\alpha_1)
\sideset{}{^*}\sum_{\alpha_2\bmod \mathfrak{q}}\chi(n_2+\alpha_2)\times\mathbf{T},
\end{split}
\end{equation}
with
\bna
\mathbf{T}&=&\sum_{m=1}^{\infty}\widetilde{V}\left(\frac{m}{\mathcal{N}_0}\right)
U_f^{\pm}\left(\frac{mNL^2}{p_1^2\mathfrak{q}^2}\right)
\overline{U_f^{\pm}\left(\frac{mNL^2}{p_2^2\mathfrak{q}^2}\right)}
\nonumber\\&&\times
e\left(\mp\frac{\overline{n_1\mathfrak{q}}m\ell_1^2}{p_1}\right)
e\left(\pm\frac{\overline{n_2\mathfrak{q}}m\ell_2^2}{p_2}\right)
e\left(\pm\frac{\overline{\alpha_1p_1}m\ell_1^2}{\mathfrak{q}}
\mp\frac{\overline{\alpha_2p_2}m\ell_2^2}{\mathfrak{q}}\right).
\ena
In the following we only consider the $+$ case,
and the other case can be analyzed similarly.
Applying Poisson summation with modulus $p_1p_2\mathfrak{q}$
to $\mathbf{T}$, we get
\begin{equation*}\label{n sum after poisson}
\begin{split}
\mathbf{T}=\frac{\mathcal{N}_0}{p_1p_2\mathfrak{q}}\sum_{m}
\sum_{b\bmod{p_1p_2\mathfrak{q}}}e\left(-\frac{\overline{n_1\mathfrak{q}}b\ell_1^2}{p_1}\right)
e\left(\frac{\overline{n_2\mathfrak{q}}b\ell_2^2}{p_2}\right)
e\left(\frac{\overline{\alpha_1p_1}\ell_1^2-\overline{\alpha_2p_2}\ell_2^2}{\mathfrak{q}}b\right)
e\left(\frac{bm}{p_1p_2\mathfrak{q}}\right)\, \CJ\left(\frac{m\mathcal{N}_0}{p_1p_2\mathfrak{q}}\right),
\end{split}
\end{equation*}
where
\begin{equation*}\label{the integral J}
\begin{split}
\mathcal{J}\left(\frac{m\mathcal{N}_0}{p_1p_2\mathfrak{q}}\right):=
\int_{\BR} \widetilde{V}(x)
U_f^{+}\left(\frac{x\mathcal{N}_0NL^2}{p_1^2\mathfrak{q}^2}\right)
\overline{U_f^{+}\left(\frac{x\mathcal{N}_0NL^2}{p_2^2\mathfrak{q}^2}\right)}
e\left(-\frac{m\mathcal{N}_0x}{p_1p_2\mathfrak{q}}\right)
\mathrm{d}x.
\end{split}
\end{equation*}
The integral $\CJ\left(\frac{m\mathcal{N}_0}{p_1p_2\mathfrak{q}}\right)$
gives arbitrarily power savings in $N$
if $|m|\gg  P^2\mathfrak{q}^{1+\varepsilon}/\mathcal{N}_0$.
Hence we can truncate the dual $m$-sum at
$|m|\ll N^{\varepsilon}P^2\mathfrak{q}/\mathcal{N}_0$ at the cost of a negligible error.
For smaller values of $m$, we use the trivial bound $\CJ\left(\frac{m\mathcal{N}_0}{p_1p_2\mathfrak{q}}\right)\ll 1$.
Since $(p_1p_2,\mathfrak{q})=1$, the character sum over $b$ factors as
\bna
\sum_{b\bmod{p_1p_2}}
e\left(\frac{(-\overline{n_1}\ell_1^2p_2 +\overline{n_2}\ell_2^2p_1+m)
\overline{\mathfrak{q}}b}{p_1p_2}\right)
\sum_{b\bmod \mathfrak{q}}
e\left(\frac{(\overline{\alpha_1}\ell_1^2p_2-\overline{\alpha_2} \ell_2^2p_1+m)
\overline{p_1p_2}b}{\mathfrak{q}}\right)\,
\ena
which vanishes until $-\overline{n_1}\ell_1^2p_2
+\overline{n_2}\ell_2^2p_1+m\equiv0(\bmod \;p_1p_2)$ and
$\overline{\alpha_1}\ell_1^2p_2-\overline{\alpha_2} \ell_2^2p_1+m\equiv 0
\bmod{\mathfrak{q}}$,
in which case it equals $p_1p_2\mathfrak{q}$. Therefore,
\begin{equation*}
\begin{split}
\mathbf{T}
=&\mathcal{N}_0\underset{\begin{subarray}{c}
|m|\ll N^{\varepsilon}P^2\mathfrak{q}/\CN_0\\ -\overline{n_1}\ell_1^2p_2
+\overline{n_2}\ell_2^2p_1+m\equiv0(\bmod p_1p_2) \\
\overline{\alpha_1}\ell_1^2p_2-\overline{\alpha_2} \ell_2^2p_1+m\equiv 0\bmod{\mathfrak{q}}\end{subarray}}{\sum}
\CJ\left(\frac{m\mathcal{N}_0}{p_1p_2\mathfrak{q}}\right)+O_A(N^{-A}).
\end{split}
\end{equation*}

Substituting the above into \eqref{before truncate},
\begin{equation*}
\begin{split}
\mathcal{S}_2^\star(N, \mathcal{N}_0)\ll &\frac{\CN_0}{P^4}\sum_{\ell_1\in\mathcal{L}}
\sum_{\ell_2\in\mathcal{L}}\sum_{p_1\in\mathcal{P}}\sum_{p_2\in\mathcal{P}}
\underset{\begin{subarray}{c}1\leq |n_1|\ll R
\\ (n_1,p_1)=1\end{subarray}}{\sum}\quad
\underset{\begin{subarray}{c}1\leq |n_2|\ll R\quad
\\ (n_2,p_2)=1\end{subarray}}{\sum}
\underset{\begin{subarray}{c}
|m|\ll P^2\mathfrak{q}^{1+\varepsilon}/\CN_0
\\ -\overline{n_1}\ell_1^2p_2 +\overline{n_2}\ell_2^2p_1+m\equiv0(\bmod p_1p_2) \end{subarray}}{\sum}
\left|\mathfrak{C}\right| + O_A\left(N^{-A}\right),
\end{split}
\end{equation*}
where
\begin{equation}\label{character sum}
\mathfrak{C}=\sideset{}{^*}\sum_{\alpha\bmod \mathfrak{q}}\overline{\chi}(n_1+\alpha)
\chi\bigg(n_2+\ell_2^2p_1(\overline{\overline{\alpha}\ell_1^2p_2+m})\bigg),
\end{equation}
and
\begin{equation*}
\begin{split}
R:=N^{\varepsilon}P\mathfrak{q}/N.
\end{split}
\end{equation*}
Then,
\bea\label{s2NN0}
\mathcal{S}_2^\star(N,\CN_0) \ll \mathcal{S}_0^{\flat}(N,\CN_0) +
\mathcal{S}_1^{\flat}(N,\CN_0) + O_A\left(N^{-A}\right),
\eea
where
\begin{equation*}
\mathcal{S}_0^{\flat}(N,\CN_0) = \frac{\CN_0}{P^4} \sum_{\ell\in\CL}
\sum_{p_1\in\mathcal{P}}\sum_{p_2\in\mathcal{P}}
\underset{\begin{subarray}{c} 1\leq |n_1|\ll R \\ (n_1,p_1)=1\end{subarray}}{\sum}\quad
\underset{\begin{subarray}{c} 1\leq |n_2|\ll R \\ (n_2,p_2)=1\end{subarray}}{\sum}\quad
\underset{\begin{subarray}{c}
|m|\ll P^2\mathfrak{q}^{1+\varepsilon}/\CN_0 \\ -\overline{n_1}\ell^2 p_2 +
\overline{n_2}\ell^2 p_1+m\equiv0(\bmod p_1p_2) \end{subarray}}{\sum} \left|\mathfrak{C}\right|,
\end{equation*}
and
\begin{equation*}
\begin{split}
\mathcal{S}_1^{\flat}(N, \mathcal{N}_0)= \frac{\CN_0}{P^4} \mathop{\sum_{\ell_1\in\CL}
\sum_{\ell_2\in\CL}}_{\ell_1\neq \ell_2} \sum_{p_1\in\mathcal{P}}
\sum_{p_2\in\mathcal{P}}\underset{\begin{subarray}{c} 1\leq|n_1|\ll R \\ (n_1,p_1)=1\end{subarray}}{\sum}\quad
\underset{\begin{subarray}{c} 1\leq|n_2|\ll R \\ (n_2,p_2)=1\end{subarray}}{\sum}\quad
\underset{\begin{subarray}{c}|m|\ll P^2\mathfrak{q}^{1+\varepsilon}/\CN_0 \\
-\overline{n_1}\ell_1^2 p_2 +\overline{n_2}\ell_2^2 p_1+m\equiv0(\bmod p_1p_2) \end{subarray}}{\sum}
\left|\mathfrak{C}\right|.
\end{split}
\end{equation*}

To estimate $\mathfrak{C}$, we quote the following results (see \cite{AHLS}).
\begin{Lemma}\label{squareroot 2}
Let $\mathfrak{q}>3$ be a prime and we define
\begin{eqnarray*}
\mathfrak{C}=\sum_{z\in \mathbb{F}_\mathfrak{q}^{\times}
\atop (m+\gamma\overline{z},\mathfrak{q})=1}\overline{\chi}(n_1+z)
\chi\bigg(n_2+\alpha(\overline{m+\gamma\overline{z}})\bigg), \qquad (\alpha\gamma,\mathfrak{q})=1,\quad m,n_1,n_2,\alpha,\gamma\in \mathbb{Z}.
\end{eqnarray*}
Suppose that $(n_1n_2,\mathfrak{q})=1$. If $\mathfrak{q}|m$, we have
\begin{eqnarray*}
\mathfrak{C}=\chi(\alpha\overline{\gamma})
R_\mathfrak{q}(n_2-n_1\alpha\overline{\gamma})-\chi(n_2\overline{n_1}),
\end{eqnarray*}
where $R_\mathfrak{q}(a)=\sum_{z\in \mathbb{F}_\mathfrak{q}^{\times}}e(az/\mathfrak{q})$ is the Ramanujan sum.
If $\mathfrak{q}\nmid  m$ and at least one of $n_1-\overline{m}\gamma$ and $n_2+\overline{m}\alpha$ is
nonzero in $\mathbb{F}_\mathfrak{q}$, then
\begin{equation*}
\mathfrak{C}\ll \mathfrak{q}^{1/2}.
\end{equation*}
Finally, if $m\neq0$ and $n_1-\overline{m}\gamma=n_2+\overline{m}\alpha=0$ in $\BF_\mathfrak{q}$, then
\[
 \mathfrak{C}=
 \begin{cases}
 -\chi(mn_2\overline{\gamma}) \quad &\text{ if } \chi \text{ is not a quadratic character}, \\
 \chi(\overline{m}n_2\gamma)(\mathfrak{q}-1) \quad &\text{ if } \chi \text{ is a quadratic character}.
 \end{cases}
\]
\end{Lemma}

Recall $\mathfrak{q}^{\beta}<N\ll \mathfrak{q}^{1+\varepsilon}$ with
$\beta>2/3$ and we will choose $P<\mathfrak{q}^{3/8}$,
so that $R<\mathfrak{q}$ and thus $(n_1n_2,\mathfrak{q})=1$.
Write
$\mathcal{S}_0^{\flat}(N,\CN_0)\ll \Delta_1 + \Delta_2$ and
$\mathcal{S}_1^{\flat}(N,\CN_0)\ll\Sigma_1 + \Sigma_2$.
The contribution of the terms with $m\equiv 0\bmod{\mathfrak{q}}$ is given by
$\Delta_1$ and $\Sigma_1$, and the contribution of the terms with
$m\not\equiv 0\bmod{\mathfrak{q}}$ is given by $\Delta_2$ and $\Sigma_2$,
with $\Delta_i$ and $\Sigma_j$ appropriately defined.

\subsection{$\underline{m\equiv 0\bmod{\mathfrak{q}}}$}
For the sum \eqref{character sum}, Lemma \ref{squareroot 2} shows that
\begin{equation*}
\mathfrak{C}= \begin{cases} O(\mathfrak{q}), & \quad \textit{if }
n_2\ell_1^2p_2\equiv n_1\ell_2^2 p_1 (\bmod \mathfrak{q})\\
O(1), & \quad \mathrm{otherwise}.\\
\end{cases}
\end{equation*}
According to $ n_2\ell_1^2p_2\equiv n_1\ell_2^2 p_1\bmod{\mathfrak{q}}$ or not, we write
\begin{equation*}
\begin{split}
\Delta_1=\Delta_{10}+\Delta_{11} \quad \text{ and } \quad \Sigma_1=\Sigma_{10}+\Sigma_{11},
\end{split}
\end{equation*}
where
\begin{equation*}
\Delta_{10} := \frac{\CN_0}{P^4} \sum_{\ell\in\CL}
\sum_{p_1\in\mathcal{P}}\sum_{p_2\in\mathcal{P}}
\mathop{\underset{\begin{subarray}{c} 1\leq |n_1|\ll R
\\ (n_1,p_1)=1\end{subarray}}{\sum}\quad
\underset{\begin{subarray}{c} 1\leq |n_2|\ll R \\ (n_2,p_2)=1\end{subarray}}{\sum}}
_{n_2p_2\equiv n_1p_1(\bmod \mathfrak{q})}\underset{\begin{subarray}{c}
|m|\ll P^2\mathfrak{q}^{1+\varepsilon}/\CN_0 \\ -\overline{n_1}\ell^2 p_2 +
\overline{n_2}\ell^2 p_1+m\equiv0(\bmod p_1p_2)
\\ m\equiv0(\bmod \mathfrak{q})\end{subarray}}{\sum} \mathfrak{q},
\end{equation*}
and
\begin{equation*}
\begin{split}
\Sigma_{10}:=\frac{\CN_0}{P^4} \mathop{\sum_{\ell_1\in\CL}
\sum_{\ell_2\in\CL}}_{\ell_1\neq \ell_2} \sum_{p_1\in\mathcal{P}}
\sum_{p_2\in\mathcal{P}}\mathop{\underset{\begin{subarray}{c} 1\leq|n_1|\ll R \\ (n_1,p_1)=1\end{subarray}}{\sum}\quad
\underset{\begin{subarray}{c} 1\leq|n_2|\ll R \\ (n_2,p_2)=1\end{subarray}}{\sum}}_
{n_2\ell_1^2p_2\equiv n_1\ell_2^2p_1(\bmod \mathfrak{q})}
\underset{\begin{subarray}{c}|m|\ll P^2\mathfrak{q}^{1+\varepsilon}/\CN_0\\
-\overline{n_1}\ell_1^2 p_2 +\overline{n_2}\ell_2^2 p_1+m\equiv0(\bmod p_1p_2)
\\ m\equiv0(\bmod\mathfrak{q})\end{subarray}}{\sum}\mathfrak{q}.
\end{split}
\end{equation*}
$\Delta_{11}$ and $\Sigma_{11}$ are the other pieces with the congruence condition
$n_2\ell_1^2p_2\nequiv n_1\ell_2^2 p_1\bmod{\mathfrak{q}}$.
We have the following estimates.
\begin{Lemma}$\Sigma_{10}\ll N^{\varepsilon}\frac{\mathfrak{q}^4}{N^2L}$
and $\Delta_{10}\ll N^{\varepsilon}\frac{\mathfrak{q}^4}{N^2L}$. \label{B10}
\end{Lemma}
\begin{proof}
We only estimate $\Sigma_{10}$. The estimation of $\Delta_{10}$
is very similar and simpler.
Suppose
\begin{equation}\label{restriction 3}
\begin{split}
P^2L^2\ll N^{1-\varepsilon}.
\end{split}
\end{equation}
Then the congruence $n_2\ell_1^2p_2\equiv n_1\ell_2^2 p_1(\bmod\,{\mathfrak{q}})$
implies that $n_2\ell_1^2p_2= n_1\ell_2^2 p_1$.
Therefore fixing $\ell_1,p_2,n_2$ fixes $\ell_2,p_1,n_1$ up to
factors of $\log \mathfrak{q}$.

Finally, the congruence conditions on $n_1, n_2$ and $m$ can be combined to write
\begin{equation*}
\begin{split}
-\overline{n_1}\ell_1^2 p_2+\overline{n_2}\ell_2^2p_1+m\equiv 0(\bmod\, p_1p_2\mathfrak{q}).
\end{split}
\end{equation*}
Since $P^2\mathfrak{q}^{1+\varepsilon}/N\ll \mathcal{N}_0$,
the $m$ sum  satisfies $|m|\ll N$,
which is smaller than the size of the modulus $p_1p_2\mathfrak{q}$.
Therefore for fixed $n_i, \ell_i, p_i$, the $m$ sum is at most singleton.
Therefore,
\begin{equation*}
\begin{split}
\Sigma_{10}\ll \frac{\mathcal{N}_0}{P^4}LPR\mathfrak{q}\ll
\frac{\mathfrak{q}^4}{N^2L}.
\end{split}
\end{equation*}
\end{proof}

\begin{Lemma}$\Sigma_{11}\ll N^\varepsilon\frac{\mathfrak{q}^4}{N^3}$ and
$\Delta_{11}\ll N^\varepsilon\frac{\mathfrak{q}^4}{N^3L}$. \label{B11}
\end{Lemma}
\begin{proof}
We only estimate $\Sigma_{11}$. The estimation of $\Delta_{11}$
is very similar and simpler.
If $p_1\neq p_2$, then $(m,p_1p_2)=1$, $n_1\equiv \overline{m}\ell_1^2p_2(\bmod{p_1})$
and $n_2\equiv - \overline{m} \ell_2^2 p_1 (\bmod{p_2})$.
These congruence conditions therefore save a factor of $P$ in each $n_i$-sum.
The congruence $m\equiv0\bmod \mathfrak{q}$
saves a factor of at most $\mathfrak{q}$ in the $m$-sum.
If $p_1=p_2=p$, then the congruence conditions imply $p|m$.
Thus $p\mathfrak{q}|m$, since we already have $\mathfrak{q}|m$. Since
$
\mathcal{N}_0>P^{1+\varepsilon},
$
then \eqref{S star final} implies $|m|\ll P^2\mathfrak{q}^{1+\varepsilon}/\CN_0 \ll P\mathfrak{q}$,
which implies that $m=0$.
The remaining congruence condition $n_1\ell_2^2\equiv n_2\ell_1^2(\bmod p)$
shows that fixing $n_1,\ell_2,\ell_1$ saves a factor of $P$ in the $n_2$-sum.
Therefore,
\begin{equation*}
\begin{split}
\Sigma_{11}\ll \frac{\mathcal{N}_0}{P^4}L^2
\bigg[P^2\left(1+\frac{P^2}{\mathcal{N}_0}\right)
\left(\frac{R}{P}\right)^2+P\cdot\frac{R^2}{P}\bigg]\ll \frac{\mathfrak{q}^4}{N^3}.
\end{split}
\end{equation*}
\end{proof}

\subsection{$\underline{m\nequiv 0\bmod{\mathfrak{q}}}$}
Lemma \ref{squareroot 2} shows that
\begin{equation*}
\mathfrak{C}= \begin{cases} O(\mathfrak{q}), & \quad \textit{if }
n_1-\overline{m}\ell_1^2p_2\equiv n_2+\overline{m}\ell_2^2p_1\equiv 0(\bmod \mathfrak{q})\\
O(\mathfrak{q}^{1/2}), & \quad \mathrm{otherwise}.\\
\end{cases}
\end{equation*}
According to $n_1-\overline{m}\ell_1^2p_2\equiv n_2+\overline{m}\ell_2^2p_1\equiv 0(\bmod \mathfrak{q})$ or not,
we write
\begin{equation*}
\begin{split}
\Delta_2 = \Delta_{20} + \Delta_{21} \quad \text{ and } \quad \Sigma_2=\Sigma_{20}+\Sigma_{21},
\end{split}
\end{equation*}
where
\begin{equation*}
\Delta_{20} := \frac{\CN_0}{P^4} \sum_{\ell\in\CL}\sum_{p_1\in\mathcal{P}}\sum_{p_2\in\mathcal{P}}
\underset{\begin{subarray}{c} 1\leq |n_1|\ll R \\ (n_1,p_1)=1\end{subarray}}{\sum}\quad
\underset{\begin{subarray}{c} 1\leq |n_2|\ll R \\ (n_2,p_2)=1\end{subarray}}{\sum}
\underset{\begin{subarray}{c}|m|\ll P^2\mathfrak{q}^{1+\varepsilon}/\CN_0\\
-\overline{n_1}\ell^2 p_2 +\overline{n_2}\ell^2 p_1+m\equiv0(\bmod p_1p_2) \\
m\not\equiv0(\bmod \mathfrak{q}) \\
n_1-\overline{m}\ell^2p_2\equiv n_2+\overline{m}\ell^2p_1\equiv 0
(\bmod \mathfrak{q})\end{subarray}}{\sum} \mathfrak{q},
\end{equation*}
and
\begin{equation*}
\begin{split}
\Sigma_{20}:=\frac{\CN_0}{P^4} \mathop{\sum_{\ell_1\in\CL}
\sum_{\ell_2\in\CL}}_{\ell_1\neq \ell_2} \sum_{p_1\in\mathcal{P}}
\sum_{p_2\in\mathcal{P}}\underset{\begin{subarray}{c} 1\leq|n_1|\ll R \\ (n_1,p_1)=1\end{subarray}}{\sum}\quad
\underset{\begin{subarray}{c} 1\leq|n_2|\ll R \\ (n_2,p_2)=1\end{subarray}}{\sum}
\underset{\begin{subarray}{c}|m|\ll P^2\mathfrak{q}^{1+\varepsilon}/\CN_0 \\
-\overline{n_1}\ell_1^2 p_2 +\overline{n_2}\ell_2^2 p_1+m\equiv0(\bmod p_1p_2) \\
m\not\equiv0(\bmod \mathfrak{q}) \\
n_1-\overline{m}\ell_1^2p_2\equiv n_2+\overline{m}\ell_2^2p_1\equiv 0
(\bmod \mathfrak{q})\end{subarray}}{\sum} \mathfrak{q},
\end{split}
\end{equation*}
$\Delta_{21}$ and $\Sigma_{21}$ are the corresponding other pieces.
We have the following estimates.
\begin{Lemma}$\Sigma_{20}\ll N^\varepsilon\frac{\mathfrak{q}^4}{N^2L}$ and
$\Delta_{20}\ll N^\varepsilon\frac{\mathfrak{q}^4}{N^2L}$. \label{B20}
\end{Lemma}
\begin{proof}
We only estimate $\Sigma_{20}$. The estimation of $\Delta_{20}$
is very similar and simpler.
The congruence conditions on $n_1, n_2$ and $m$ can be combined to write
\begin{equation*}
-\ell_1^2p_2+mn_1\equiv0(\bmod p_1\mathfrak{q}) \quad \text{ and } \quad
\ell_2^2p_1+ mn_2\equiv0(\bmod p_2\mathfrak{q}).
\end{equation*}
If $\mathcal{N}_0\gg P^2\mathfrak{q}^{1+\varepsilon}/N$,
we have $|mR|\ll P\mathfrak{q}$.
The congruence conditions therefore give equalities
\begin{equation*}
\begin{split}
m n_1 = \ell_1^2 p_2 \quad \text{ and } \quad m n_2 = -\ell_2^2 p_1,
\end{split}
\end{equation*}
Therefore fixing $\ell_1,p_2$ fixes $m,n_1$ up to factors
of $\log \mathfrak{q}$. Moreover, for fixed $m$ and $n_2$,
$\ell_2$ and $p_1$ are fixed.
Therefore,
\begin{equation*}
\begin{split}
\Sigma_{20}\ll \frac{\mathcal{N}_0}{P^4}LPR\mathfrak{q}\ll \frac{\mathfrak{q}^4}{N^2L}.
\end{split}
\end{equation*}
\end{proof}

\begin{Lemma}$\Sigma_{21}\ll N^\varepsilon\left(\frac{\mathfrak{q}^{9/2}}{N^3}
+\frac{\mathfrak{q}^{7/2}L^2}{N^2}\right)$ and
$\Delta_{21}\ll N^\varepsilon\left(\frac{\mathfrak{q}^{9/2}}{N^3L}
+\frac{\mathfrak{q}^{7/2}L}{N^2}\right)$. \label{B21}
\end{Lemma}
\begin{proof}
We only estimate $\Sigma_{21}$. The estimation of $\Delta_{21}$
is very similar and simpler.
When $p_1\neq p_2$,
the congruence $-\overline{n_1}\ell_1^2 p_2+\overline{n_2}
\ell_2^2p_1+m\equiv 0(\bmod p_1p_2)$ implies
that $(m,p_1p_2)=1$. Moreover, for fixed $m$, $p_i$ and $\ell_i$, $i=1,2$,
\begin{equation*}
\begin{split}
n_1\equiv \overline{m}\ell_1^2p_2(\bmod{p_1} )\quad \text{ and } \quad
n_2\equiv -\overline{m}\ell_2^2p_1(\bmod{p_2}).
\end{split}
\end{equation*}
These congruence conditions save a factor of $P$ in each $n_i$-sum.
In case $p_1=p_2=p$, the congruence condition shows $p|m$.
Moreover,  $-\overline{n_1}\ell_1^2 +\overline{n_2}
\ell_2^2+m/p\equiv 0(\bmod p)$.
Therefore, for fixed $n_1$, $\ell_1$, $\ell_2$ and $m$, we can save $P$
in the $n_2$-sum.
\begin{equation*}
\begin{split}
\Sigma_{21}\ll  \frac{\mathcal{N}_0\mathfrak{q}^{1/2}}{P^4}L^2
\bigg[P^2\left(\frac{R}{P}\right)^2\left(1+\frac{P^2\mathfrak{q}}{\CN_0}\right)+P\frac{R^2}{P}
\left(1+\frac{P\mathfrak{q}}{\CN_0}\right)\bigg] \ll
N^\varepsilon\left(\frac{\mathfrak{q}^{9/2}}{N^3}+\frac{\mathfrak{q}^{7/2}L^2}{N^2}\right).
\end{split}
\end{equation*}
\end{proof}

\subsection{Conclusion}
Lemmas \ref{B10}, \ref{B11}, \ref{B20} and \ref{B21} and \eqref{s2NN0} imply
\begin{equation*}
\begin{split}
\mathcal{S}_2^\star(N,\CN_0)\ll N^\varepsilon\left(\frac{\mathfrak{q}^4}{N^2L}
+\frac{\mathfrak{q}^{7/2}L^2}{N^2}+\frac{\mathfrak{q}^{9/2}}{N^3}\right).
\end{split}
\end{equation*}
Inserting this into \eqref{S star final}, we get
\begin{equation*}
\mathcal{S}_2^\star(N)\ll N^\varepsilon\left(\mathfrak{q}^{3/4+\varepsilon}
+ \frac{\mathfrak{q}^{1/2}N^{1/2}}{L^{1/2}}
+ N^{1/2}\mathfrak{q}^{1/4}L\right).
\end{equation*}
Similarly,
\begin{equation*}
\mathcal{S}_2(c=\mathfrak{q})\ll \frac{N^\varepsilon}{P}\left(\mathfrak{q}^{3/4+\varepsilon}
+ \frac{\mathfrak{q}^{1/2}N^{1/2}}{L^{1/2}}
+ N^{1/2}\mathfrak{q}^{1/4}L\right).
\end{equation*}
These estimates combined with \eqref{S(N) and S(N) star} yield
\begin{equation*}
\frac{\mathcal{S}_2^\sharp(N)}{N^{1/2}}\ll N^\varepsilon
\left(\frac{\mathfrak{q}^{3/4+\varepsilon}}{N^{1/2}}
+ \frac{\mathfrak{q}^{1/2}}{L^{1/2}} +\mathfrak{q}^{1/4}L
+\frac{P}{N^{1/2}}\right).
\end{equation*}
We will take $P=N^{1+\varepsilon}L^2/\mathfrak{q}<\mathfrak{q}^{3/8}$. Thus
the first term dominates the last term and
\begin{equation*}
\frac{\mathcal{S}_2^\sharp(N)}{N^{1/2}}\ll\mathfrak{q}^\varepsilon
\left(\frac{\mathfrak{q}^{3/4+\varepsilon}}{N^{1/2}}
+ \frac{\mathfrak{q}^{1/2}}{L^{1/2}} +\mathfrak{q}^{1/4}L\right).
\end{equation*}
By the assumption
$\mathfrak{q}^{\beta}<N\leq \mathfrak{q}^{1+\varepsilon}$,
the second term dominates the first term if we take $L<\mathfrak{q}^{\beta-1/2}$.
Moreover, by the assumption in \eqref{restriction 3}, we assume $L<\mathfrak{q}^{1/6-\varepsilon}$.
Then the assumption $P=N^{1+\varepsilon}L^2/\mathfrak{q}\ll\mathfrak{q}^{3/8}$ is
guaranteed. We conclude that for
$\mathfrak{q}^{(1-\beta)/2+\varepsilon}\ll L\ll \mathfrak{q}^{1/6-\varepsilon}$,
\begin{equation*}
\frac{\mathcal{S}_2^\sharp(N)}{N^{1/2}}\ll\mathfrak{q}^\varepsilon
\left(\frac{\mathfrak{q}^{1/2}}{L^{1/2}} +\mathfrak{q}^{1/4}L\right).
\end{equation*}
Here the lower bound of $L$ comes from the assumption in \eqref{restriction 1}.
That proves Proposition \ref{cusp form prop}.

\begin{Remark}
For $\mathcal{S}_1^\sharp(N)$, the restriction on $L$ will be
$ \mathfrak{q}^{1-\beta+\varepsilon}\ll
L\ll \mathfrak{q}^{\min\{\beta-1/2,1/3\}-\varepsilon}$.
\end{Remark}

\subsection*{Acknowledgments.}
We thank Yongxiao Lin for many valuable suggestions.
The authors are partially supported by National Natural Science Foundation
of China (Grant No. 11871306).

\end{document}